\documentclass{amsart}
\usepackage{amsfonts}
\usepackage{graphicx}
\usepackage{tabularx}
\usepackage{array}
\usepackage[usenames,dvipsnames]{color}
\usepackage{comment}
\usepackage{hyperref}
\usepackage{amsmath}
\usepackage{amsthm}
\usepackage{amssymb}
\usepackage{fullpage}
\usepackage[shortlabels]{enumitem}
\usepackage[usenames,dvipsnames]{xcolor}
\usepackage{tikz}

\tikzstyle{legend_general}=[rectangle, rounded corners, thin,
                          top color= white,bottom color=lavander!25,
                          minimum width=2.5cm, minimum height=0.8cm,
                          violet]

\newtheorem{theorem}{Theorem}[section]

\newtheorem{lemma}[theorem]{Lemma}

\theoremstyle{definition}
\newtheorem{definition}[theorem]{Definition}

\newtheorem{remark}[theorem]{Remark}

\title{From one to many rainbow Hamiltonian cycles}

\author{Peter Bradshaw}
\address{Department of Mathematics, Simon Fraser University, Vancouver, Canada}
\email{pabradsh@sfu.ca}

\author{Kevin Halasz}
\address{Department of Mathematics, Simon Fraser University, Vancouver, Canada}
\email{khalasz@sfu.ca}

\author{Ladislav Stacho}
\address{Department of Mathematics, Simon Fraser University, Vancouver, Canada}
\email{lstacho@math.sfu.ca}

\begin{document}
\begin{abstract}
Given a graph $G$ and a family $\mathcal{G} = \{G_1,\ldots,G_n\}$ of subgraphs of $G$, a transversal of $\mathcal{G}$ is a pair $(T,\phi)$ such that $T \subseteq E(G)$ and $\phi: T \rightarrow [n]$ is a bijection satisfying $e \in G_{\phi(e)}$ for each $e \in T$. We call a transversal Hamiltonian if $T$ corresponds to the edge set of a Hamiltonian cycle in $G$. We show that, under certain conditions on the maximum degree of $G$ and the minimum degrees of the $G_i \in \mathcal{G}$, for every $\mathcal{G}$ which contains a Hamiltonian transversal, the number of Hamiltonian transversals contained in $\mathcal{G}$ is bounded below by a function of $G$'s maximum degree. This generalizes a theorem of Thomassen stating that, for $m \geq 300$, no $m$-regular graph is uniquely Hamiltonian. We also extend Joos and Kim's recent result that, if $G = K_n$ and each $G_i \in \mathcal{G}$ has minimum degree at least $\frac{n}{2}$, then $\mathcal{G}$ has a Hamiltonian transversal: we show that, in this setting, $\mathcal{G}$ has exponentially many Hamiltonian transversals. Finally, we prove analogues of both of these theorems for transversals which form perfect matchings of $G$. 
\end{abstract}
\maketitle
\section{Introduction}

Suppose we have a graph $G$ and a family $\mathcal G = \{G_1, \dots, G_s\}$ of (not necessarily distinct) subgraphs of $G$. Then
we define a $\mathcal G$-\emph{transversal}  as a pair $(T,\phi)$ where $T \subseteq E(G)$ and $\phi: T \rightarrow [s]$  is a  bijection satisfying $e \in G_{\phi(e)}$ for each $e \in T$. Informally, a $\mathcal G$-transversal  is a set of edges that uses exactly one edge from each subgraph $G_i \in \mathcal G$
or, in other words, a system of distinct representatives for the collection $\mathcal G$. One may also think of a $\mathcal G$-transversal by giving the edges of each graph in $\mathcal G$ a single color distinct from the colors used on the other graphs and then defining a $\mathcal G$-transversal as a set of edges in which each color appears exactly once. In this sense, the notion of a $\mathcal G$-transversal is stronger than the well-studied notion of a \emph{rainbow subgraph}, which would correspond to a collection using at most one edge from each graph $G_i \in \mathcal G$.

The formal definition of a $\mathcal G$-transversal was introduced by Joos and Kim in \cite{JoosKim2020}, though the idea appeared implicitly in earlier work on 
rainbow triangles \cite{Aharoni,Magnant2015} and rainbow perfect matchings (see e.g. \cite{Aharonietal2017,BaratWankess2014,Pokrovskiy2018}). The main result of \cite{JoosKim2020} was a generalization of Dirac's theorem \cite{Dirac}, a classical result giving a sufficient condition for the existence of a Hamiltonian cycle in a graph. Joos and Kim proved
that if $G = K_n$ and $\mathcal G = \{G_1,G_2,\ldots,G_n\}$ satisfies $\delta(G_i) \geq \frac{n}{2}$ for each $i \in [n]$, then there exists a \emph{Hamiltonian} $\mathcal G$-transversal, i.e.\! a $\mathcal G$-transversal whose constituent edges form a Hamiltonian cycle on $V(G)$. In a similar fashion, the first author showed in \cite{BradshawBipancyclicity} that a similar condition of Moon and Moser \cite{Moon} for Hamiltonicity in bipartite graphs can be generalized into the setting of transversals.

This paper continues the project of generalizing results related to Hamiltonicity into the language of graph transversals. In \cite{Thomassen}, Thomassen proves the following theorem.

\begin{theorem}
\label{thmThomassen}
If $G$ is a Hamiltonian $m$-regular graph, where $m \geq 300$, then $G$ contains at least two Hamiltonian cycles.
\end{theorem}

Thomassen proves Theorem \ref{thmThomassen} by 2-coloring the edges of $G$, using red to color the edges of a Hamiltonian cycle $C$ and using green to color all other edges. Thomassen had previously shown in \cite{ThomassenRedGreen} that if $G$ has a vertex subset that is independent with respect to the red edges and dominating with respect to the green edges, then $G$ must contain a second Hamiltonian cycle. To complete the proof of Theorem \ref{thmThomassen}, Thomassen uses the Lov\'asz Local Lemma \cite{LLL} to show that a red-independent and green-dominating vertex subset of $G$ must exist under the given hypotheses. Using Thomassen's ideas, we will prove a generalization of Theorem \ref{thmThomassen} for graph transversals.

\begin{theorem}
Let $G$ be an $n$-vertex graph of maximum degree $m$, where $m \geq 262$, and let $\mathcal{G}~=~\{G_1, \dots, G_n\}$ be a family of subgraphs of $G$, each of minimum degree at least $7\sqrt{m \log m}$. If $\mathcal G$ contains a Hamiltonian transversal, 
then $\mathcal G$ contains at least $\left \lceil \frac{1}{60}\log m \right \rceil!$ distinct Hamiltonian transversals.
\label{thmMain}
\end{theorem}

The method that we use to prove Theorem \ref{thmMain} may also be used to show that, given any $\epsilon>0$, there is a value $m_0$ such that for $m>m_0$, a minimum degree of $(6 + \epsilon) \sqrt{m \log m}$ is enough to guarantee the existence of many Hamiltonian transversals.

When $m \geq 300$, the minimum degree $t$ required in Theorem \ref{thmMain} is less than $m$. Therefore, Theorem \ref{thmThomassen} may be obtained from Theorem \ref{thmMain} by considering an $m$-regular graph $G$ and letting each graph of $\mathcal{G}$ be equal to $G$. Furthermore, in \cite{Horak}, Hor\'ak and the third author showed that in Thomassen's proof of Theorem \ref{thmThomassen}, it is not necessary for $G$ to be entirely regular and that Thomassen's method allows some gap between the minimum and maximum degree of $G$.  Theorem \ref{thmMain} also generalizes the main result of \cite{Horak}. More recently, Haxell, Seamone, and Verstraete \cite{HSV07} extended Theorem \ref{thmThomassen} to the case $m \geq 23$, while noting that Thomassen's methods could in fact work for any $m \geq 73$. It is unclear whether Theorem \ref{thmMain} could be extended to similarly small values of $m$. 

For Dirac graphs (that is, graphs on $n$ vertices whose minimum degree is at least $\frac{n}{2}$), we establish an even stronger result. In \cite{JoosKim2020}, Joos and Kim prove that given a family $\mathcal G$ of $n$ Dirac graphs on a common set of $n$ vertices, $\mathcal G$ must contain a Hamiltonian transversal. By applying Thomassen's methods from \cite{ThomassenRedGreen} and \cite{Thomassen}, we strengthen the result of Joos and Kim.

\begin{theorem} 
 For every $\epsilon>0$ and every $c \geq \frac{1}{2}$, there exists an $n_0$ such that, for all $n \geq n_0$, a family $\mathcal{G} = \{G_1, \dots, G_n\}$ of subgraphs of $K_n$ in which each $G_i$ has minimum degree at least $cn$ contains at least $\left\lceil \frac{c^2 n}{16+\epsilon}  \right\rceil!$ Hamiltonian transversals.
\label{thmDiracHam}
\end{theorem}

We also consider transversals which are perfect matchings. Specifically, given a graph $G$ with $2n$ vertices and family of subgraphs $\mathcal{G}=\{G_1,G_2,\ldots,G_n\}$, we refer to a $\mathcal{G}$-transversal which forms a perfect matching on $V(G)$ as a \textit{perfect matching transversal}. Although perfect matchings in graphs can be found in polynomial time using Edmonds' blossom algorithm \cite{Edmonds}, the existence problem for perfect matching transversals is known to be NP-complete even if we restrict to the case where $G$ is bipartite \cite{ItaietalPMTransversalNPComplete}. Perfect matching transversals are a special case of the well-studied notion of matchings in uniform hypergraphs. Indeed, we may define a 3-uniform hypergraph $H(\mathcal{G})$ on the vertex set $V(G) \cup [n]$ by adding the edge $uvi$ whenever $uv \in G_i$. Notice that a perfect matching transversal of $\mathcal{G}$ corresponds to a perfect matching in $H(\mathcal{G})$. The literature concerning the existence of perfect matchings in 3-uniform hypergraphs is extensive, and very powerful technical tools have been utilized to show that various degree conditions are sufficient (see, for example, \cite{Kwan20,LM14,ZZL18}). 

We are here interested in finding lower bounds for the number of perfect matching transversals in graph families meeting certain conditions. Such questions are considered less often in the literature, but there has been some very interesting recent work in the case where $G = K_{n,n}$ and each $G_i$ is a matching. Perarnau and Serra \cite{PS13} used probabilistic tools to show that, when each $G_i$ is not too large, the number of perfect matching transversals is roughly $c^n n!$ for some small constant $c$. More recently, a landmark pair of papers by Eberhard, Manners and Mrazovi\'{c} \cite{EMM19,EMM20pre} asymptotically enumerated the perfect matching transversals when the vertices of $K_{n,n}$ are labelled by the elements of a specified group of order $n$ and each $G_i$ is the perfect matching corresponding to edges whose endpoints have the same product. 

Our framework is quite different from the just-mentioned work, as we are considering each $G_i$ to have minimum degree bounded well away from 1. Nonetheless, the existence of these results motivates our proving a perfect matching analogue of Theorem \ref{thmMain}.

\begin{theorem}%c=0.1
Let $G$ be a $2n$-vertex graph of maximum degree $m$, where $m \geq 44$, and let $\mathcal{G}~=~\{G_1, \dots, G_n\}$ be a set of subgraphs of $G$, each of minimum degree
at least $10\log(m)+6$.
If $\mathcal{G}$ contains a perfect matching transversal, then $\mathcal G$ contains at least $\left\lceil \frac{1}{2}\log m\right\rceil !$ distinct perfect matching transversals.
\label{thmPM}
\end{theorem}

In \cite{JoosKim2020}, Joos and Kim also show that, given a family $\mathcal G$ of $n$ Dirac graphs on a common set of $2n$ vertices, $\mathcal G$ must contain a perfect matching transversal. By applying Thomassen's methods from \cite{ThomassenRedGreen} and \cite{Thomassen}, we strengthen this result as well.

\begin{theorem}
For every $\epsilon>0$ and every $c \geq \frac{1}{2}$, there exists an $n_0$ such that, for all $n \geq n_0$, a family $\mathcal{G} = \{G_1, \dots, G_n\}$ of subgraphs of $K_{2n}$ in which each $G_i$ has minimum degree at least $cn$ contains at least $\left\lfloor\frac{cn}{2+\epsilon} \right\rfloor !$ distinct perfect matching transversals.
\label{thmDiracPM}
\end{theorem}

It is worth contrasting the lower bound in Theorem \ref{thmDiracPM} to the best known upper bound for the number of perfect matching transversals. A result of Tanarenko \cite{Ta17} concerning perfect matchings in an $r$-uniform hypergraphs tells us that, in the setting of Theorem \ref{thmDiracPM}, there are fewer than $(2n)^{5n/3}$ perfect matching transversals. In contrast, we show that there are at least $\left(\frac{2}{11}n\right)^{n/3}$ perfect matching transversals. The gap between these bounds could be partially explained by the fact that Tanarenko's result describes arbitrary 3-uniform hypergraphs, whereas the hypergraphs $H(\mathcal{G})$ corresponding to our subgraph families always have an independent set on one-third of the vertices. Nonetheless, it is possible that the true number of perfect matching transversals in graph classes with bounded minimum degree is significantly larger than our 
lower bounds.

The paper will be organized as follows. In Section \ref{secProb}, we establish two probabilistic tools on which we will rely for many of our proofs. In Section \ref{secHam}, we will consider Hamiltonian transversals, and we will prove Theorems \ref{thmMain} and \ref{thmDiracHam}. In Section \ref{secPM}, we will consider perfect matching transversals, and we will prove Theorems \ref{thmPM} and \ref{thmDiracPM}. Finally, in Section \ref{secConclusion}, we will consider some open questions.

\section{Probabilistic tools}
\label{secProb}
Given a natural number $n$, let $[n]:=\{1,2,\ldots,n\}$. As many of our proofs are probabilistic, we will need two nontrivial probabilistic tools. First, we will use the following asymmetric version of the Lov\'asz Local Lemma 
given in \cite{Graham}. 
\begin{theorem}
Let $A_1, \dots, A_n$ be events in a probability space. Let $G$ be a graph with vertex set $A_1, \dots, A_n$ such that for each $i \in [n]$, $A_i$ is independent of any combination of events that are not neighbors of $A_i$ in $G$. Suppose there exist positive real numbers $x_1, \dots, x_n$, each less than $1$, such that for each $i \in [n]$, we have
$$\Pr(A_i) < x_i \prod(1 - x_j),$$
where the product is taken over all $j$ for which $A_j$ is a neighbor of $A_i$. Then 
$$\Pr(\overline{A_1} \land \dots \land \overline{A_n}) \geq \prod_{i = 1}^n (1 - x_i) > 0.$$
\label{LLL}
\end{theorem}

Next, we will need the following forms of the Chernoff bound, which appear e.g.~in Chapter 4 of \cite{Mitzenmacher}.  
\begin{theorem}
\label{thmChernoff}
Let $Y$ be a random variable that is the sum of pairwise independent indicator variables and let $\mu$ be the expected value of $Y$. For any value $0 < \delta < 1$, the following inequalities hold:
\begin{eqnarray}
\Pr(Y < (1-\delta)  \mu) &\leq &\left (  \frac{e^{-\delta}}{(1-\delta)^{1-\delta}} \right )^{\mu} , \label{eq:chernoff1}\\
\Pr(Y < (1-\delta)  \mu) &\leq& \exp \left ( - \frac{1}{2} \delta^2 \mu \right ).\label{eq:chernoff2}
\end{eqnarray}
\end{theorem}

\section{Many Hamiltonian transversals}
\label{secHam}
Throughout most of this section, we will consider a fixed graph $G$ on the vertex set $V =  \{x_1, \dots, x_n\}$ and a fixed family $\mathcal G = \{G_1, \dots, G_n\}$ of subgraphs of $G$. Throughout the following we consider indices modulo $n$, so that (for example) $x_{n+1} = x_1$ and $x_0 = x_n$.
Our goal will be to show that under appropriate conditions, if $\mathcal G$ contains a Hamiltonian transversal, then $\mathcal{G}$ contains many distinct Hamiltonian transversals.
So, we will assume that $\mathcal{G}$ has a fixed Hamiltonian transversal $(C,\phi)$. By reordering indices if necessary we may assume that $C = (x_1 x_2, x_2 x_3, \dots, x_{n-1}x_n, x_n x_1)$ and that for each $i \in [n]$, we have $\phi(x_i x_{i+1}) = i$. When $C$ satisfies these indexing conditions, we say that $\mathcal{G}$ is \textit{naturally indexed} with respect to $(C,\phi)$.

\begin{definition}  Assuming $\mathcal G$ is naturally indexed with respect to $(C,\phi)$, we say that $H$ is the \emph{full RYB-digraph} of $(\mathcal{G},C,\phi)$, denoted by $H_{\mathcal{G},C,\phi}$, if $V(H) = V$ and the following three conditions hold for each $i\in[n]$:
\begin{itemize}
    \item $H$ has bidirectional red edges $x_i x_{i+1}$ and $x_i x_{i+2}$,
    \item $H$ has a yellow arc $x_i x_j$  whenever $x_i x_j \in E(G_i)$ and $j \not \in \{i-1, i+1\}$,
    \item $H$ has a blue arc $x_i x_j$ whenever $x_i x_j \in E(G_{i-1})$ and $j \not \in \{i-1, i+1\}$.
\end{itemize}
More generally, we refer to any spanning subgraph of $H_{\mathcal{G},C,\phi}$ which contains all of the red edges of $H_{\mathcal{G},C,\phi}$ as an \emph{RYB-digraph} of $(\mathcal{G},C,\phi)$. 
 \label{defRYB}
  \end{definition}
Because the graph $G$ and the family $\mathcal{G}$ remain fixed throughout most of this section (before changing to another fixed family in Subsection \ref{subsec:hamcomplete}), we may write $H_{C,\phi}$ for the full RYB-digraph of $(\mathcal{G},C,\phi)$, and more generally talk of an RYB-digraph of $(C,\phi)$, with no risk of ambiguity.  Also, at times we will want to ignore arc directions in an RYB-digraph $J$; we write $\overline{J}$ for the undirected graph underlying $J$ (where arcs may retain their color while losing direction). 

Given an RYB-digraph $J \subseteq H_{C,\phi}$, we use $R(J)$, $Y(J)$, and $B(J)$ to denote, respectively, the sets of red edges, yellow arcs, and blue arcs of $J$. For a vertex $v \in V$ and any set of edges $F \subseteq \binom{V}{2}$, we use $N_{F}(v)$ to denote the set of vertices $w \in V$ such that $vw \in F$. Similarly, for any set of arcs $A \subseteq V \times V$, we use $N^+_{A}(v)$ to denote the set of vertices $w \in V$ such that $vw \in A$. 
We say that a set $S \subseteq V$ is \emph{red-independent} with respect to $(C,\phi)$ if no two vertices in $S$ are adjacent by an edge in $R(H_{C,\phi})$. Notice that, because $R(J) = R(H_{C,\phi})$ for all RYB-digraphs of $(C,\phi)$, the notion of red-independence does not depend on any specific $J$. We say that a nonempty set $S \subseteq V$ is \emph{locally $J$-dominating} if for every vertex $x_i \in S$, the sets $N^+_{Y(J)}(x_{i-1}) \cap S$ and $N^+_{B(J)}(x_{i+1}) \cap S$ are nonempty, omitting the prefix $J$ when it is clear from context. An example of a red-independent locally dominating set in an RYB-digraph is given in Figure \ref{figJ}. 

\begin{figure}[t]

\begin{tikzpicture}
[scale=3.5,auto=left,every node/.style={circle,fill=gray!30}]

\draw[red, ultra thick] (0,0) circle [thick, radius=1];
\node (z) at (-0.33,-0.75) [fill = white]  {$j-1$};
\node (z) at (0.3,0.75) [fill = white]  {$1$};
\node (z) at (-0.3,0.75) [fill = white]  {$n$};
\draw[yellow] [->, line width=0.5mm] (-0.25,0.97) -- (-0.18,0.6);
\draw[blue] [->, line width=0.5mm] (0.25,0.97) -- (0.2,0.6);
\node (z) at (0,1.15) [fill = white]  {$x_1$};
\node (z) at (0.28,1.12) [fill = white]  {$x_2$};
\node (z) at (-0.28,1.12) [fill = white]  {$x_n$};
%\draw[red,ultra thick] (-0.25,0.97) arc (210:330:0.3);

\node (z) at (0.25,0.97) [draw = black]  {};
\node (z) at (-0.25,0.97) [draw = black]  {};

\node (z) at (1.1,0.3) [fill = white]  {$x_i$};
\node (z) at (0.62,0.48) [fill = white]  {$i - 1$};
\draw[yellow] [->, line width=0.5mm] (0.86,0.5) -- (0.5,0.3);
\node (z) at (0.86,0.5) [draw = black]  {};

\node (z) at (0.8,0.11) [fill = white]  {$i $};
\draw[blue] [->, line width=0.5mm] (0.997,0.07) -- (0.6,0.01);
\node (z) at (0.997,0.07) [draw = black]  {};

\draw[blue] [->,line width=0.5mm] (0.752,-0.652) -- (0.45,-0.4);
\node (z) at (0.752,-0.652) [draw = black]  {};

\draw[yellow] [->, line width=0.5mm] (0.95,-0.32) -- (0.58,-0.17);
\node (z) at (0.95,-0.32) [draw = black]  {};

\draw[yellow] [->, line width=0.5mm] (-0.55,-0.83) -- (-0.36,-0.51);
\node (z) at (-0.55,-0.83) [draw = black]  {};

\node (z) at (-0.83,-0.82) [fill = white]  {$x_j$};
\node (z) at (-0.72,-0.35) [fill = white]  {$j$};

\draw[blue] [->, line width=0.5mm] (-0.83,-0.55) -- (-0.55,-0.35);
\node (z) at (-0.83,-0.55) [draw = black]  {};

\draw[yellow] [->, line width=0.5mm] (-0.95,0.3) -- (-0.6,0.17);
\node (z) at (-0.95,0.3) [draw = black]  {};

\draw[blue] [->, line width=0.5mm] (-0.75,0.66) -- (-0.47,0.43);
\node (z) at (-0.75,0.66) [draw = black]  {};

\filldraw (0,1) circle (1.8pt);
\filldraw (0.955,0.3) circle (1.8pt);
\filldraw (0.866,-0.5) circle (1.8pt);
\filldraw (-0.7071,-0.7071) circle (1.8pt);
\filldraw (-0.866,0.5) circle (1.8pt);
\end{tikzpicture} 
\caption{The RYB-digraph $J'$ of $(\mathcal{G},C,\phi)$ used in the proof of Theorem \ref{lemmaPMcond}, along with a red-independent locally-dominating set $S \subseteq V$. The vertices of $S$ are shown in black. Each vertex adjacent to $S$ by an edge in $C$ has one incident yellow or blue arc. Each yellow arc and blue arc has an endpoint in $S$ not shown in the figure. 
}
\label{figJ}
\end{figure}
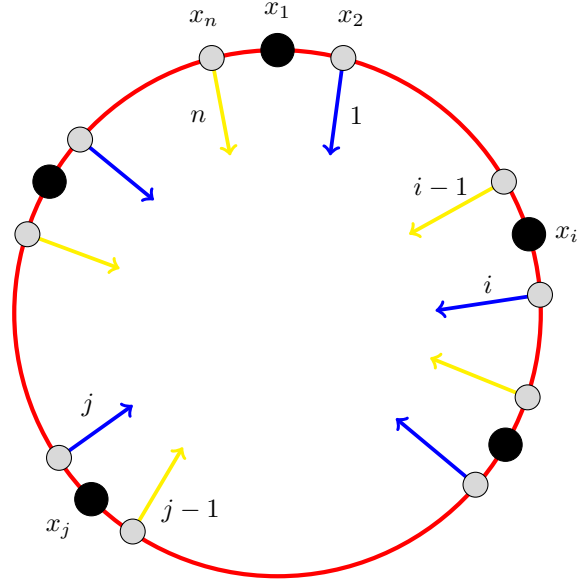

\subsection{Technical conditions} \label{sectionRYB}

Our first result gives sufficient conditions for the existence of a second Hamiltonian transversal in a graph family that already contains a Hamiltonian transversal. Given a red-independent set $S \subseteq V$, write $\Omega(C,\phi,S)$ for the set of Hamiltonian transversals $(C', \psi)$ such that $C'$ contains all $|S|$ of the paths induced by $C$ on the vertices $V \setminus S$, and such that for all $e \in C \cap C'$ we have $\phi(e) = \psi(e)$, and for all $v \in N_{C}(S)$ the edge $e$ joining $v$ to $S$ in $C$ and the edge $e'$ joining $v$ to $S$ in $C'$ satisfy $\phi(e) = \psi(e')$. 

\begin{lemma}\label{lemmaIndDom}
Let $J$ be an RYB-digraph of $(C,\phi)$. If there exists a nonempty set $S \subseteq V$ that is red-independent and locally $J$-dominating, then $\Omega(C,\phi,S)$ contains a second Hamiltonian transversal, distinct from $(C,\phi)$, comprised of edges in $E(\overline{J}) \cap E(G)$.
\end{lemma}
\begin{proof}
By permuting the indices of $x_1, \dots, x_n$, we may assume that $\mathcal{G}$ is naturally indexed with respect to $(C,\phi)$ and that $x_1 \in S$. We will find a second Hamiltonian transversal in $\Omega(C,\phi,S)$ which uses only edges underlying a subgraph $J^\prime \subseteq J$
satisfying $E(\overline{J^\prime}) \subseteq E(\overline{J}) \cap E(G)$. To obtain $J'$, first remove all edges in $R(J) \setminus C$. Then, remove yellow and blue arcs until $J^\prime$ has, for each $x_i \in S$, exactly one yellow arc joining $x_{i-1}$ to $S$ and exactly one blue arc joining $x_{i+1}$ to $S$ (such arcs must exist in $J^\prime$ by the hypothesis that $S$ is locally-dominating). This means that for every $v \in V \setminus S$, if $v$ is adjacent in $C$ to a vertex in $S$, then $v$ has degree 3 in $\overline{J^\prime}$, and otherwise $v$ has degree 2 in $\overline{J^\prime}$.

Our first aim is to show that $\overline{J^\prime}$ contains at least two Hamiltonian cycles. To this end, we define an auxiliary graph $A$ whose vertex set is the set of Hamiltonian paths in $\overline{J^\prime}$ beginning with the edge $x_1 x_2$. We define adjacency in $A$ as in Thomason's lollipop argument (introduced in \cite{Thomason}): let $Q = (x_1, x_2, v_3, \dots, v_n)$ be adjacent to $Q' = (x_1, x_2, v_3, \dots, v_{j-1}, v_j, v_n, v_{n-1}, \dots, v_{j+1})$ whenever $v_jv_n \in E(\overline{J^\prime})$. Observe that the induced subgraph of $\overline{J^\prime}$ on the vertices $V \setminus S$ consists of exactly $s:= |S|$ paths $P_1 \dots, P_s$, so that every Hamiltonian path  in $\overline{J^\prime}$ beginning with the edge $x_1 x_2$ has the form 
\begin{equation}x_1, P_{1}, x_{j_1}, P_{i_1}, \ldots,x_{j_{s-1}}, P_{i_{s-1}},
\label{eq:hampathform}
\end{equation}
for some permutation $\left(i_1 \, i_2 \, \ldots \, i_{s-1}\right)$ of the set $[2,s]$ and some  $\{j_1 , j_2 , \ldots , j_{s-1}\} \subseteq [n]$. As each endpoint of each $P_i$ is adjacent in $C$ to a vertex in $S$, every Hamiltonian path in $Q \in V(A)$ ends in a vertex with degree 3 in $\overline{J^\prime}$. Thus, $\deg_A(Q) \in \{1,2\}$, with a path $Q = (x_1, x_2, v_3, \dots, v_n)$ having degree $2$ in $A$ if and only if $v_n$ is not adjacent to $x_1$ in $\overline{J^\prime}$. As the path $Q_0 = (x_1, x_2, x_3, \dots, x_n)$ has degree $1$ in $A$ and $\sum_{Q \in V(A)} \deg_A(Q) \equiv 0 \pmod{2}$, there must exist another path $Q^* = (x_1, x_2, y_3, \dots, y_n)$, distinct from $Q_0$, of degree $1$ in $A$. Because $y_nx_1 \in E(\overline{J^\prime})$, the vertex sequence defining $Q^\ast$ defines a Hamiltonian cycle in $G$, say $C^\ast$, which is distinct from $C$, uses only edges in $E(\overline{J^\prime})$, and contains all $|S|$ paths induced by $C$ in $V \setminus S$. 

It now suffices to define a valid bijection $\phi^\ast:E(C^\ast) \rightarrow [n]$. Recall that $Y(J^\prime)$ and $B(J^\prime)$ consist of directed edges while $R(J^\prime) = \{x_ix_{i+1} \, : \, i \in [n]\}$ is comprised of bidirectional edges. We begin by defining an injection $\psi:E(Q^\ast) \rightarrow [n]$: given $x_ix_j \in E(Q^\ast)$,
\begin{itemize}
    \item if $x_ix_{i+1}\in R(J^\prime)$, then set $\psi(x_ix_{i+1}) :=i$,
    \item if $x_ix_j \in Y(J^\prime)$, then set $\psi(x_ix_j) := i$, and
    \item  if $x_i x_j \in B(J^\prime)$, then set $\psi(x_i x_j) := i-1$.
\end{itemize} 
Observe that this definition satisfies $\phi(e) = \psi(e)$ for all $e \in C \cap Q^\ast$ and that, for all $v \in N_{C}(S)\setminus \{y_n\}$, the edge $e$ joining $v$ to $S$ in $C$ and the edge $e'$ joining $v$ to $S$ in $Q^\ast$ satisfy $\phi(e) = \psi(e')$. 
 Moreover, it follows directly from Definition \ref{defRYB} that $e \in G_{\psi(e)}$ for each $e \in E(Q^\ast)$.
 
 To see that $\psi$ is injective, first notice that,

 the edges in $E(Q^\ast) \cap R(J^\prime)$ must be mapped to distinct values by $\psi$. Because the vertices in $S$ are at a mutual distance of at least $3$ in $J^\prime$ and every vertex in $V \setminus S$ is incident to at most one arc in $Y(J^\prime) \cup B(J^\prime)$, the edges in $E(Q^\ast) \cap (Y(J^\prime) \cup B(J^\prime))$ are also mapped to distinct values by $\psi$. Now, suppose  that there exist $x_ix_{i+1} \in E(Q^\ast) \cap R(J^\prime)$ and $x_kx_\ell \in E(Q^\ast) \cap Y(J^\prime)$ satisfying $\psi(x_ix_{i+1}) = \psi(x_kx_\ell)$. This is only possible if $k=i$ and $x_\ell,x_{i+1} \in S$, in which case $x_\ell, x_i, x_{i+1}$ defines a 2-edge subpath of $Q^\ast$. This means $Q^\ast$ visits two vertices in $S$ without traversing one of the paths formed by the components of $\overline{J^\prime}[V \setminus S]$ in between, contradicting the fact that all Hamiltonian paths in $\overline{J^\prime}$ beginning with $x_1x_2$ have the form given by \eqref{eq:hampathform}. Using a similar argument to show that there are no $e \in E(Q^\ast) \cap R(J^\prime)$ and $f \in E(Q^\ast) \cap B(J^\prime)$ with $\psi(e) = \psi(f)$, we may conclude that the function $\psi:E(Q^\ast) \rightarrow [n]$ is injective.

Because $Q^*$ must have the form given by \eqref{eq:hampathform}, $y_n$ is the endpoint of a red path in the graph $\overline{J^\prime}[V \setminus S]$. It follows that $y_n$ is incident in $\overline{J^\prime}$ to two red edges and either a yellow edge or a blue edge. In either case, Definition \ref{defRYB} tells us that, among the three edges incident to $y_n$ in $\overline{J^\prime}$, the red edge not in $E(Q^\ast)$ and the non-red edge must both lie in some $G_{i} \in \mathcal{G}$. In fact, because $J^\prime$ has at most one yellow or blue arc emanating from each vertex, these two edges are the only edges in $E(\overline{J^\prime}) \cap E(G_i)$. Thus, $c$ must be the unique color in $[n] \setminus \psi(E(Q^\ast))$. Defining $\phi^\ast$ by $\phi^\ast(y_nx_1) := c$ and $\phi^\ast(e) := \psi(e)$ for $e \in E(Q^\ast)$ gives us the desired bijection. The result follows from the observation that the colors we assigned to $Q^*$ by $\psi$ are such that $(C^\ast,\phi^\ast) \in \Omega(C, \phi, S)$.
\end{proof}

It is worth clearly noting a couple of easily verified properties of Hamiltonian transversals in $\Omega(C,\phi,S)$.
\begin{remark}
\label{remark:stayyy}
For every $(C',\psi) \in \Omega(C,\phi,S)$
\begin{enumerate}[(a)]
\item if $e \in C$ is not incident to a vertex of $S$, then $e$ belongs to $C'$, and
\item the vertices of $S$ are at a mutual distance of at least $3$ in $C'$.
\end{enumerate}
\end{remark}

Assuming $\mathcal G$ is naturally indexed with respect to $(C,\phi)$, for every red-independent set $S \subseteq V$  
 we define
$$d_{C,\phi}^*(S) := \min_{x_i \in S}\{|N^+_{Y(H_{C,\phi})}(x_{i-1}) \cap S|,|N^+_{B(H_{C,\phi})}(x_{i+1}) \cap S|\}.$$ 
Notice that here we are taking the minimum over $2|S|$ different values. We claim that, for every Hamiltonian transversal $(C',\psi) \in \Omega(C,\phi,S)$,
\begin{equation}
d^\ast_{C,\phi}(S) = d^\ast_{C',\psi}(S).
\label{eq:dinvariant}\end{equation}
Indeed, given $x_i \in S$, we know by Remark \ref{remark:stayyy} that $x_{i-2}x_{i-1} \in C'$. Let $v \in S$ denote the other neighbor of $x_{i-1}$ in $C'$. Notice that we could have $v = x_i$ and, if not, we have $x_{i-1}v \in Y(H_{C,\phi})$ and $x_{i-1}x_i \in Y(H_{C'})$. This means that the set of vertices in $S$ which are joined to $x_{i-1}$ by an edge in $G_{i-1}$ can be expressed in two ways: as $(N^+_{Y(H_{C,\phi})}(x_{i-1}) \cap S) \cup \{x_i\} $ and as  $(N^+_{Y(H_{C',\psi})}(x_{i-1}) \cap S) \cup \{v\} $. We may then conclude that $|N^+_{Y(H_{C,\phi})}(x_{i-1}) \cap S| = |N^+_{Y(H_{C',\psi})}(x_{i-1}) \cap S|$. The claim then follows from the essentially identical argument that $|N^+_{B(H_{C,\phi})}(x_{i+1}) \cap S| = |N^+_{B(H_{C',\psi})}(x_{i+1}) \cap S|$.

We are now ready to prove our main technical result.

\begin{theorem}
Let $S \subseteq V$ be a red-independent set. If $d^*_{C,\phi}(S) \geq d$,
then $\Omega(C,\phi,S)$ contains at least $(d+1)!$ distinct Hamiltonian transversals.
\label{thmHamOut}
\end{theorem}
\begin{proof}

We prove this result by induction on $d$. Notice that the case $d=1$ follows directly from Lemma \ref{lemmaIndDom}.

Assuming $d \geq 2$, we claim that there exists a vertex $v_0 \in N_C(S)$, say $v_0s_0 \in C$, for which the following property holds:
\begin{equation}\label{eqmanycond}
   \text{Every edge in $E(G_{\phi(v_0 s_0)}) $ joining $v_0$ to $S$ belongs to a Hamiltonian transversal in $\Omega(C,\phi,S)$.} 
\end{equation} 
Supposing the claim is false, 
for each $v \in N_C(S)$, write $s_v \in S$ so that $vs_v \in E(C)$, 
and choose
an edge $vw = e_v \in E(G_{\phi(vs_v)}) $ such that $w \in S$ and 
$e_v$ is not in any Hamiltonian transversal of $\Omega(C,\phi,S)$. 
Letting $F := \{e_v \,:\, v \in N_C(S)\}$, notice that the definition of $e_v$ implies $F \cap C = \emptyset$. Consider the RYB-digraph $J$ comprised of the red edges in $H_{C,\phi}$ together with yellow and blue arcs corresponding to the edges in $F$. Observe that $J$ has been constructed so that $S$ is locally $J$-dominating. Lemma \ref{lemmaIndDom} then tell us that there is a Hamiltonian transversal $C' \in \Omega(C,\phi,S)$, distinct from $C$, such that $E(C') \subseteq E(\overline{J}) \cap E(G)$.  But $F = (E(\overline{J}) \cap E(G)) \setminus E(C)$, so the fact that $C' \neq C$ implies $C$ must use an edge in $F$, contradicting the definition of $F$ as comprised of edges contained in no element of $\Omega(C,\phi,S)$. This proves the claim.

 Let  $v_0 \in V \setminus S$ denote a vertex, adjacent in $C$ to $s_0 \in S$, satisfying \eqref{eqmanycond}. Fix an edge $e \in G_{\phi(v_0 s_0)}$ and let  $(C',\psi)$ denote the Hamiltonian transversal in $\Omega(C,\phi,S)$ containing $e$ guaranteed to exist by the claim in the previous paragraph. Denote by $s_1$ the unique vertex in $S$ incident to $v_0$ in $C'$ and let $S' = S \setminus \{s_1\}$. To apply induction to $S'$, we must show that $d_{C',\psi}^\ast(S') \geq d-1$. Indeed, we know from  \eqref{eq:dinvariant} that $d_{C',\psi}^\ast(S) \geq d$, while for each $x \in N_C(S)$ we have 
 $$N^+_{Y(H_{C',\psi})}(x) \cap S \subseteq (N^+_{Y(H_{C',\psi})}(x) \cap S') \cup \{s_1\} \text{ and }N^+_{B(H_{C',\psi})}(x) \cap S \subseteq (N^+_{B(H_{C',\psi})}(x) \cap S') \cup \{s_1\}.$$  
Thus, we see that $d_{C',\psi}^\ast(S') \geq d-1$, and we can apply the induction hypothesis 
to find a set $\mathcal{T}_e$ of $d!$ distinct Hamiltonian transversals. Observe that, by Remark \ref{remark:stayyy}, for each $C'' \in \mathcal{T}_e$, the two edges incident to $v_0$ in $C''$ are $e$ and the unique edge of $C$ joining $v_0$ to $V \setminus S$.

Now, since $d_{C,\phi}^\ast(S) \geq d$, there are at least $d+1$ edges in $G_{\phi(v_0s_0)}$ joining $v_0$ to a vertex of $S$,
so, we have at least $d+1$ choices for our fixed edge $e \in G_{\phi(v_0 s_0)}$.
For each such $e$ we get a set $\mathcal{T}_e$ of at least $d!$ Hamtilonian transversals. Moreover, by considering the pair of edges incident to $v_0$ in the various cycles, we see that $\mathcal{T}_e \cap \mathcal{T}_{e'} = \emptyset$ whenever $e \neq e'$.

This gives us a total of at least $(d+1) \cdot d! = (d+1)!$ distinct Hamiltonian transversals of $\mathcal G$. 
\end{proof}

\subsection{Proof of Theorem \ref{thmMain}}

In Theorem \ref{thmHamOut} we showed how to use a Hamiltonian transversal and a vertex set $S$ with particularly nice properties to find many Hamiltonian transversals. Now, we use the Lov\'{a}sz Local Lemma to show that, so long as our base graph $G$ has sufficiently large maximum degree and each $G_i \in \mathcal{G}$ has sufficiently large minimum degree, we may find such a set $S$. From there, it will be straightforward to complete the proof of Theorem \ref{thmMain}.

\begin{theorem}
Let $H=H_{C,\phi}$ and let $m$ and $r$ be positive integers satisfying $m \geq 262$ and $r \geq 7\sqrt{m \log m}+2$. If $\Delta(G) = m$  and each $v \in V$ satisfies $\min\{N^+_{B(H)}(v),N^+_{Y(H)}(v)\} \geq r$,  
then there exists a red-independent set $S \subseteq V$ 
satisfying
$d_{C,\phi}^*(S) \geq  \frac{r}{400}\sqrt{\frac{\log m}{m}}.$
\label{LLLresult}
\end{theorem}

\begin{proof}
By permuting indices we may assume that $\mathcal{G}$ is naturally indexed with respect to $(C,\phi)$. Let $p = \frac{1}{2}\sqrt{\frac{\log m}{m}}$. We will randomly construct a set $S \subseteq V$ by independently adding each vertex of $V$ to $S$ with probability $p$, then show that $S$ satisfies the condition given in the theorem statement with positive probability.

For each bidirectional edge $e \in R(H)$, let $A_e$ be the event that both endpoints of $e$ are added to $S$. We refer such an event an \textit{$x$-event}. Observe that, for each $e \in R(H)$,
$$\Pr(A_e) = p^2.$$
For each vertex $v \in V$, let $A_v$ to be the event that 
$|N^+_{Y}(v) \cap S| < \frac{pr}{400}$ and let $A_v'$ to be the event that 
$|N^+_{B}(v) \cap S| < \frac{pr}{400}$.  
We refer to such events as \textit{$y$-events}. By the Chernoff bound \eqref{eq:chernoff1}, for each $v \in V$ 
$$\Pr(A_v) \leq  \xi ^{pr} \text{ and } \Pr(A_v') \leq \xi ^{pr}, \text{ where } \xi = \frac{e^{-\frac{399}{400}}}{\left(\frac{1}{400}\right)^{\frac{1}{400}}}.$$
Notice that if no $x$-event and no $y$-event occurs, then $S$ is a non-empty red-independent set satisfying $d_{C,\phi}^*(S)~\geq~\frac{r}{800}\sqrt{\frac{\log m}{m}}$, as desired.

Each $y$-event corresponds to a monochromatic star $T$ in $H$ in which each arc directed away from a ``center'' vertex. Because $G$ has maximum degree $m$ and, by Definition \ref{defRYB}, neither of the red edges in $E(H) \cap E(G)$ incident to a vertex $v \in V$ can underlie a blue or yellow arc, so each $y$-event is uniquely determined by a color (either blue or yellow) and a vertex (corresponding to the center of $T$).

To apply the local lemma, we count the dependencies between each of our events. First, consider an $x$-event $A_e$. Because each vertex is incident to four red edges, $A_e$ is dependent with at most $6$ other $x$-events. Moreover, the star $T$ corresponding to any $y$-event dependent with $A_e$ must contain one of the endpoints of $e$. We have two choices for the color of $T$ and, for each color, at most $2(m-1)$ choices for its center. Thus, the number of $y$-events dependent with $A_e$ is at most $4m-4$.

Next, consider a $y$-event $A_v$ and let $T_0$ denote the corresponding yellow star. Any $x$-event dependent with $A_v$ must correspond to a red edge incident to a vertex in $T_0$, so $A_v$ depends on at most $4(m-1)$ $x$-events. If $T$ is the star corresponding to a $y$-event dependent with $A_v$ then $V(T) \cap V(T_0) \neq \emptyset$.  As both $T$ and $T_0$ have at most $m-1$ vertices, we have at most $(m-1)^2$ choices for the center of $T$. Because we have two choices for the color of $T$, we see that $A_v$ is dependent with at most $2(m-1)^2$ $y$-events. Similarly, every $y$-event of the form $A_v'$ is dependent with at most $4m-4$ $x$-events and at most $2(m-1)^2$ $y$-events.

Therefore, if we assign a real number $x$ to all $x$-events and a real number $y$ to all $y$-events, then Theorem \ref{LLL} tells us that the following pair of inequalities imply that with positive probability no $x$-event occurs and no $y$-event occurs:
\begin{eqnarray}
\label{firstIn}
0 &<& x (1-x)^{6}(1-y)^{4m-4}-p^2; \\
\label{secondIn}
\xi^{pr}& <& y(1-x)^{4m-4}(1-y)^{2(m-1)^2}.
\end{eqnarray}
Let $x = 1.05p^2$ and $y = \frac{1}{m^2}$. Under these conditions, the right hand side of (\ref{firstIn}) approaches $\frac{p^2}{20}>0$ from above as $m$ grows arbitrarily large. Indeed, one may check that it is positive for $m \geq 262$. Looking to (\ref{secondIn}), we see that as $m$ gets arbitrarily large, the left hand side approximately approaches $m^{-3.43}$ from below, while the right hand side of approaches 
$$\frac{1}{m^2} m^{-1.05} e^{-2} = \frac{m^{-3.05}}{e^2}$$
from above. Indeed, one may check that this equation holds for all $m\geq194$.
Therefore, if $m \geq 262$, then we avoid all bad events with positive probability, and hence there must exist a red-independent set $S \subseteq V$  for which $d_{C,\phi}^*(S) \geq \frac{r}{400}\sqrt{\frac{\log m}{m}}$.
\end{proof}

We now finish the proof of Theorem \ref{thmMain}. In fact, we may prove the following stronger result, from which Theorem \ref{thmMain} follows directly.

\begin{theorem}
Let $m \geq 262$ be an integer, let $G$ be an $n$-vertex graph  of maximum degree $m$, and let $\mathcal{G}~=~\{G_1, \dots, G_n\}$ be a family of subgraphs of $G$, each of minimum degree $t \geq 7\sqrt{m \log m}$. If $\mathcal G$ contains a Hamiltonian transversal, 
then $\mathcal G$ contains at least $\left \lfloor \frac{t-2}{400}\sqrt{\frac{\log m}{m}} + 1 \right \rfloor!$ distinct Hamiltonian transversals.
\end{theorem}

\begin{proof}
Let $C$ be a Hamiltonian transversal of $\mathcal{G}$ and assume without loss of generality that $\mathcal G$ is naturally indexed with respect to $(C,\phi)$. Every vertex $x_i \in V$ is incident to at least $t$ edges of $G_i$, at most two of which belong to the set $\{x_i x_{i-1}, x_i x_{i+1}\}$. Therefore, $x_i$ is the tail of at least $t-2$ yellow arcs in $H_{C,\phi}$. Similarly, each vertex $x_i$ is the tail of at least $t-2$ blue arcs in $H_{C,\phi}$. Therefore, for $m \geq 262$, $H$ satisfies the conditions of Theorem \ref{LLLresult}, and we can find a red-independent set $S \subseteq V$ such that $d_{C,\phi}^*(S) \geq \frac{t-2}{400}\sqrt{\frac{\log m}{m}}$. The result then follows by Theorem \ref{thmHamOut}.
\end{proof}

\subsection{When the base graph is complete}\label{subsec:hamcomplete}

For the rest of this section we let $\mathcal G = \{G_1, \dots, G_n\}$ denote a family of subgraphs of $K_n$ in which each graph $G_i \in \mathcal{G}$ has minimum degree at least $n/2$. In \cite{JoosKim2020}, Joos and Kim show that every such family has a Hamiltonian transversal. For the rest of this section, let $(C,\phi)$ denote a fixed Hamiltonian transversal, whose existence is given by Theorem 1 of \cite{JoosKim2020}, and let $H:=H_{\mathcal{G},C,\phi}$. Notice that this is precisely the setting of Theorem \ref{thmDiracHam}.  Our proof of this theorem relies heavily upon Theorem \ref{thmHamOut}, with the remaining difficult work being the following probabilistic proof that we can find a vertex set satisfying the hypotheses of that theorem.

\begin{lemma}\label{lemmaInS}
Suppose that each $G_i \in \mathcal{G}$ has minimum degree at least $cn$ for some $c \geq\frac{1}{2}$. For sufficiently large $n$, then there exists a red-independent set $S \subseteq V$ such that 
$$d_{C,\phi}^\ast(S) \geq \frac{c^2 }{16}n - \frac{15c^2}{8} \sqrt{n \log n} .$$
\end{lemma}
\begin{proof}
Let $H = H_{C,\phi}$. We construct $S$ randomly in two steps: (step 1) add each vertex $v \in V(G)$ to $S$ with probability $p = \frac{c}{8}$, then (step 2) remove from $S$ every pair $v,w \in S$ satisfying $vw \in R(H)$. Notice that this process is guaranteed to create a red-independent set.

After the first step, the expected size of $N^+_{Y(H)}(v) \cap S$ is at least $\mu = \frac{c^2n}{8}$ for every $v \in V$. Let $ \epsilon = 10\sqrt{ \frac{\log n}{n} }$ and note that, so long as $n \geq 650$, we have $\epsilon < 1$. By a Chernoff bound \eqref{eq:chernoff2}, the probability that $v$ has fewer than $(1 - \epsilon)\mu $ yellow neighbors in $S$ is at most $e^{ - \frac{1}{2} \epsilon^2 \mu }$.  Similarly, the probability that $v$ has fewer than $(1 - \epsilon)\mu $ blue neighbors in $S$ is at most $e^{- \frac{1}{2} \epsilon^2 \mu }$.

In the second step, the expected number of removed vertices is at most $2p^2 |R(H)| = 4np^2$. Therefore, by Markov's inequality, the probability that we delete more than $(1+\epsilon)4np^2$ vertices from $S$ in this way is at most 
$\frac{1}{1 + \epsilon} < 1 - \frac{1}{2} \epsilon.$ 
Applying the union bound, we see that every $v \in V$ satisfies 
\begin{equation}\label{eqDirac}\min\{|N^+_{Y(H)}(v) \cap S|,|N^+_{B(H)}(v) \cap S|\} \geq (1 - \epsilon) \frac{c^2n}{8} - (1+\epsilon)4np^2\end{equation}
with probability at least $\frac{1}{2} \epsilon - 2ne^{- \frac{1}{2} \epsilon^2 \mu } \geq \frac{1}{\sqrt{n}}(5\sqrt{\log(n)}-2) $, which is positive for all $n \geq 1$. Plugging in the values given above for $c$, $p$, and $\epsilon$ into the right-hand side of \eqref{eqDirac} and, we get the desired lower bound.
\end{proof}

It is worth noting that, although the proof of Lemma \ref{lemmaInS} works for $n \geq 650$, for $n\leq 8100$ the lower bound it gives for $d^\ast_{C,\phi}(S)$ is negative, meaning the result in these cases is trivial. However, as $n$ grows arbitrarily large, this bound approaches $\frac{c^2}{16}n$. %
From here, it is easy to complete the proof of
Theorem \ref{thmDiracHam}. For any $\epsilon>0$, we may use Lemma \ref{lemmaInS}   to find a red-independent set $S \subseteq V$ such that  $d_{C,\phi}^\ast(S) \geq  \frac{c^2n}{16+\epsilon}$ so long as $n$ is sufficiently large. That $\mathcal{G}$ has at least 
$\left\lceil \frac{c^2 n}{16+\epsilon}  \right\rceil!$
Hamiltonian transversals 
then follows by Theorem \ref{thmHamOut}.

\section{Perfect matching transversals}
\label{secPM}

Turning our attention to perfect matching transversals, we will henceforth use $G$ to denote a graph on $2n$ vertices. As above, $\mathcal G = \{G_1, \dots, G_n\}$ is a family of subgraphs of $G$ and $V = V(G)$. We now assume that $\mathcal G$ has a perfect matching transversal, which we denote by $(M,\phi)$, where $M$ is a matching and $\phi$ is an associated bijection between $M$ and $[n]$. We say that $\mathcal{G}$ is \emph{naturally indexed} with respect to $M$ if, setting labels $V = \{x_1,\ldots,x_n,y_1,\ldots,y_n\}$, we have $M = \{x_1y_1, x_2y_2, \dots, x_ny_n\}$ with $x_iy_i \in E(G_i)$ for every $i \in [n]$.

\begin{definition}
Assuming $\mathcal{G}$ is naturally indexed with respect to $M$, we say that $H$ is the \emph{full RB-digraph} of $(\mathcal{G},M,\phi)$, denoted $H_{\mathcal{G},M,\phi}$, if $V(H)= V$ and the following three conditions hold for each $i \in [n]$:
\begin{itemize}
    \item $H$ has a red bidirectional edge $x_i y_i$, 
    \item $H$ has a blue arc $x_iy_j$ whenever $x_iy_j \in E(G_i)$ and $j \neq i$, 
    \item $H$ has a blue arc $y_ix_j$ whenever $y_ix_j \in E(G_i)$ and $j \neq i$.
\end{itemize} 
More generally, we refer to any spanning subgraph of $H_{\mathcal{G},M,\phi}$ which contains all of the red edges of $H_{\mathcal{G},M,\phi}$ as an \emph{RB-digraph} of $(\mathcal{G},M,\phi)$.
\label{defRBgraph}\end{definition}
As we will treat $\mathcal{G}$ as fixed throughout most of this section, we will often omit the $\mathcal{G}$ in the subscript of $H_{M,\phi} = H_{\mathcal{G},M,\phi}$. Similarly, we will talk of RB-digraphs of $(M,\phi)$, omitting the $\mathcal{G}$. Given an RB-digraph $J \subseteq H_{M,\phi}$, we use $R(J)$ and $B(J)$, respectively, to denote the red edges and blue arcs in $J$. As above, we write $\overline{J}$ for the undirected graph underlying $J$. We call a set $S \subseteq V$ \emph{red-independent} if no two vertices of $S$ are adjacent via a red edge in $H_{M,\phi}$. Notice that we need not distinguish between various RB-digraphs of $(M,\phi)$ in discussing red-independence as all RB-digraphs have the same set of red edges. 

\subsection{Technical conditions} 

As in the previous section, we begin with a technical condition which guarantees the existence of two perfect matching transversals. In this case, however, the condition will be much easier to establish. Given a maximal red-independent set $S$, let $\Omega(M,\phi,S)$ denote the set of perfect matching transversals $(M',\psi)$ such that all edges in  $M'$ have one endpoint in $S$ and one endpoint in $V \setminus S$ and, for each $v \in S$, the edge $e$ in $M$ incident to $v$ and the edge $e'$ in $M'$ incident to $v$ satisfy $\phi(e) = \psi(e')$. Notice that, because $S$ is maximal, we have $|S|=|V \setminus S| = n$.

\begin{lemma}
\label{lemmaPMcond}
Let $S \subseteq V$ be a maximal red-independent set and let $J$ be an RB-digraph of $(M,\phi)$. If $N^+_{B(J)}(v) \setminus S \neq \emptyset$ for every $v \in S$, then there is an $(M',\psi) \in \Omega(M,\phi,S)$, distinct from $(M,\phi)$, comprised of edges from $E(\overline{J})$.
\end{lemma}
\begin{proof}
The work of this proof amounts to finding an even directed cycle $C$ in $J$ that alternates between red and blue arcs. Given such a cycle $C$, the symmetric difference of $M$ and $C$, which we denote by $M^\prime$, is the edge set of the second perfect matching transversal we seek. Indeed, it is easy to check that $M^\prime$ is a perfect matching between $S$ and $V \setminus S$. Furthermore, for each $s \in S$ and each edge $sv \in M \setminus M'$, there is exactly one edge $su \in M' \setminus M$ obtained from the blue arc pointing out of $s$. By Definition \ref{defRBgraph}, $su \in G_{\phi(sv)}$, so defining $\psi$ to be consistent with $\phi$ at each vertex in $S$ yields $(M',\psi) \in \Omega(M,\phi,S)$.

To find $C$, we construct a bipartite digraph $P\subseteq J$ one edge at a time; our goal is to find $C$ as a subgraph of $P$. We begin by choosing an arbitrary red edge $u_0v_0$ with $v_0 \in S$ and add $u_0v_0$ to $P$. In general, each time a red edge $uv$, with $v \in S$, is added to $P$, we then add a blue arc $vw$ with $w \not \in S$ to $P$. Notice that such an arc is guaranteed to exist by the hypothesis that each $v\in S$ has a blue out-neighbor outside of $S$. If the red edge incident to $w$ has already been added to $P$, then $P$ must contain an even cycle that alternates between red edges and blue edges. Otherwise, add to $P$ the red edge incident to $w$ and return to the previous step. As $J$ is a finite graph, this process must ultimately terminate, and hence we will eventually find an even cycle $C$ in $J$ that alternates between red edges and blue edges.
\end{proof}

Let $H = H_{M,\phi}$. For every maximal red-independent set $S\subseteq V$, we define
\[d_{M,\phi}^\times(S) := \min_{v \in S} \{| N_{B(H)}^+(v) \setminus S|\}.\]
We claim that, for every $(M',\psi) \in \Omega(M,\phi,S)$, we have
\begin{equation}\label{eq:PMswitch}
    d_{M,\phi}^\times(S) = d_{M',\psi}^\times(S).
\end{equation}
Indeed, let $C$ denote the symmetric difference of $M$ and $M'$ and observe that $C$ is a collection of even cycles alternating between edges of $M$ and edges of $M'$. Letting $H'=H_{M',\psi}$, we want to show that $|N^+_{B(H)}(v) \setminus S| = |N^+_{B(H')}(v) \setminus S|$ for every $v \in V(C) \cap S$. Given such a vertex $v$, the red edge incident to $v$ in $H$, say $vu$, becomes a blue arc pointing away from $v$ in $H'$. Similarly, the blue arc pointing out of $v$ in $C$, say $vw$, becomes a red edge in $H'$. Notice that $w \not\in S$ by construction and, because $v \in S$, $uv \in R(H)$, and $S$ is red-independent, $u \not\in S$. So, in passing from $H$ to $H'$, $v$ loses exactly one blue out-neighbor in $V \setminus S$, namely $w$, while gaining one exactly one blue out-neighbor in $V \setminus S$, namely $u$.

The astute reader may notice that, following the structure of Section \ref{secHam}, we are now set to establish an analogue of Theorem \ref{thmHamOut} for perfect matching transversals.

\begin{theorem}
Let $S \subseteq V$ be a maximal red-independent set. If  $d_{M,\phi}^\times(S) \geq d$, then $\mathcal G$ contains at least $(d + 1)!$ distinct perfect matching transversals.
\label{thmOutdeg}
\end{theorem}
\begin{proof}
We prove this result by induction on $n = |M|$. Notice that the case $n=1$ trivially holds.

In the case $n \geq 2$, we claim that there exists a vertex $v_0 \in S$ such every blue arc $v_0w$, where $w \not \in S$, belongs to a perfect matching transversal in $\Omega(M,\phi,S)$. Indeed, suppose for the sake of contradiction that for each vertex $x \in S$, there exists a blue arc $xy$ in $H=H_{M,\phi}$ such that $y \not \in S$ and $xy$ does not belong to a perfect matching transversal of $\Omega(M,\phi,S)$. Let $A$ denote a set consisting of one such arc $xy$ for each vertex $x \in S$ and let $J$ be the RB-digraph of $(M,\phi)$ satisfying $B(J) = A$.  By Lemma \ref{lemmaPMcond} there is a perfect matching transveral $(M',\psi) \in \Omega(M,\phi,S)$, distinct from $(M,\phi)$, which uses only edges in $\overline{J}$. But the only edges in $E(\overline{J}) \setminus M$ correspond to arcs in $A$, so $M'$ must contain at least one edge in $\overline{A}$. This contradicts the fact that $A$ is comprised of edges not in any perfect matching transversal, proving the claim.

This allows us to consider a fixed $v_0 \in S$ such every blue arc $v_0w$, where $w \not \in S$, belongs to a perfect matching transversal.
Given a blue out-neighbor of $v_0$, say $w_0 \not \in S$, let $(M',\psi)$ denote the perfect matching transversal that contains the edge $v_0w_0$. Let $c = \psi(v_0w_0)$, let $G^\ast$ denote the induced subgraph of $G$ on $V \setminus \{v_0,w_0\}$, and let $\mathcal{G}^\ast$ denote the corresponding corresponding family of induced subgraphs of each $G_i$ for $i \in [n]\setminus\{c\}$. Moreover, let $S^\ast = S \setminus \{v_0\}$, let $M^\ast = M'\setminus \{v_0w_0\}$, and let $\psi^\ast = \psi \vert_{M^\ast}$
. Observe that $(M^\ast,\psi^\ast)$ is a perfect matching transversal, $S^\ast$ is a maximal red-independent set with respect to $(M^\ast,\psi^\ast)$, and letting $H^\ast = H_{\mathcal{G}^\ast,M^\ast,\phi^\ast}$ it follows from \eqref{eq:PMswitch} that $|N^+_{B(H^\ast)}(v) \setminus S^\ast| \geq |N^+_{B(H)}(v) \setminus S| -1$ for each $v \in S^\ast$. Thus, $d^\times_{M^\ast,\phi^\ast}(S^\ast) \geq d-1$, and applying induction we see that $\mathcal{G}^\ast$ has at least $d!$ distinct perfect matching transversals. Moreover, each perfect matching transversal of $\mathcal{G}^\ast$, say $(L,\sigma)$, corresponds to a distinct perfect matching transversal of $\mathcal{G}$ obtained by adding $v_0w_0$ to $L$ and setting $\sigma(v_0w_0) = c$. Thus, we have a set $\mathcal{T}_{w_0}$ of at least $d!$ perfect matching transversals of $\mathcal{G}$ containing the edge $v_0w_0$.

Now, because $d^\times_{M,\phi}(S) \geq d$, the number of edges in $G_{\phi(v_0w_0)}$ joining $v_0$ to $V \setminus S$ is at least $d+1$. Thus, we have at least $d+1$ sets $\mathcal{T}_w$ of distinct perfect matching transversals and, by considering the edge incident to $v_0$ in each matching, we see that $\mathcal{T}_w \cap \mathcal{T}_{w'}=\emptyset$ whenever $w \neq w'$. Thus gives us a total of $(d+1)\cdot d!=(d+1)!$ distinct perfect matching transversals of $\mathcal{G}$. 
\end{proof}

\subsection{Proofs of Theorem \ref{thmPM} and Theorem \ref{thmDiracPM}}

 Our proof of Theorem \ref{thmPM} uses probabilistic techniques similar to those used in the proof of Theorem \ref{thmMain} via Theorem \ref{LLLresult}. We begin with a probabilistic existence theorem, then prove a more technical, but stronger, theorem from which Theorem \ref{thmPM} follows as a corollary. Unlike in Section \ref{secHam},  Theorem \ref{thmDiracPM} will in fact also follow as a corollary.

\begin{theorem}
Consider $\alpha \in (0,1)$ and $r,m \in \mathbb{N}$ such that $r \geq \frac{4(1 + \log(2m^2 - 2m + 1))}{(1-\alpha)^2}$. If $\Delta(G)=m$ and each $v \in V$ satisfies $N^+_{B(H_{M,\phi})}(v) \geq r$, then there is a maximal red-independent set $S$ such that $d_{M,\phi}^\times(S) \geq \frac{\alpha r}{2}$. 
\label{thmLLLPM}
\end{theorem}

\begin{proof}
By permuting indices, we may assume that $\mathcal{G}$ is naturally indexed with respect to $M$. We will randomly construct $S$ as follows: Independently for each edge $x_i y_i\in M$, we pick one of $x_i$ or $y_i$ uniformly at random to add to $S$. We claim that, after completing this process, there is a positive probability that in $H=H_{M,\phi}$, each vertex $v \in S$ has at least one blue out-neighbor $w \not \in S$.
For each $i$, 
we let $\{z_i\} = \{x_i,y_i\} \cap S$.
Then, let $B_i$ denote the bad event $|N^+(z_i) \setminus S| < \frac{1}{2} \alpha r$.
We will separately consider the edges $x_jy_j$ such that $\{x_j,y_j\} \subseteq N^+_{B(H)}(z_i)$, as in this case $\{x_j,y_j\}$ contributes exactly one vertex to $N^+(z_i) \setminus S$ no matter which of $x_j$ or $y_j$ we add to $S$. To this end, we define
\[Q_i = \{x_jy_j \in M \,\vert\, \{x_j,y_j\}\subseteq N_{B(H)}^+(z_i)\} ,\]
set $q_i = |Q_i|$,
and define 
$s_i = |N^+_{B(H)}(z_i) \setminus (Q_i \cup S)|.$
Then,
$|N^+_{B(H)}(z_i) \setminus S| = q_i + s_i$. 
Therefore, $B_i$ occurs if and only if  $s_i < \frac{1}{2}\alpha r - q_i$. 
So, in determining an upper bound for the probability that $B_i$ occurs, we may assume without loss of generality that $q_i < \frac{1}{2}\alpha r$, as otherwise it is guaranteed that $B_i$ will not occur. 

Now, the value of $q_i$ depends on whether $z_i=x_i$ or $z_i=y_i$. However, we can determine an upper bound which is independent of this choice. Observe that, for a fixed choice of $z_i$, the random variable $|N^+_{B(H)}(z_i) \setminus (Q_i \cup S)|$ is binomially distributed with mean at least $\frac{1}{2}r - q_i$. 
Thus, writing
$$\frac{1}{2}\alpha r - q_{i} = \epsilon_i \left (\frac{1}{2}r - q_{i} \right ),$$
where $\epsilon_i = \frac{\frac{1}{2}\alpha r - q_{i}}{\frac{1}{2}r - q_{i}}$, the Chernoff bound \eqref{eq:chernoff2} implies that $B_i$ occurs with probability at most
$$ \rho(q_i) := \exp\left(-\frac{1}{2}(1-\epsilon_i)^2\left(\frac{r}{2}-q_i\right)\right).$$

However, one may calculate
$$\frac{\partial \rho}{\partial q_i} = - \frac{(\alpha -1)^2r^2\exp\left(-\frac{(\alpha -1)^2r^2}{4(r-2q_i)}\right)}{2(r-2q_i)^2}
< 0,$$
so that $\rho$ decreases as $q_i$ increases. This implies that $\rho(q_i)$ is maximized at $q_i=0$, in which case $\epsilon_i=\alpha$. So, no matter whether $z_i=x_i$ or $z_i=y_i$, the probability 
of the bad event $B_i$ is at most
$$\exp\left(-\frac{r(1-\alpha)^2}{4}\right).$$
Because $B_i$ and $B_j$ are codependent if and only if there is a path of length at most two between $\{x_i,y_i\}$ and $\{x_j,y_j\}$, the fact that $\Delta(G) \leq m$ tells us that each bad event $B_i$ is codependent with at most $2m^2-2m$ other bad events. By Theorem \ref{LLL}, if we assign a value $x$ to each bad event, then the following inequality guarantees that with positive probability, no bad event occurs:
$$\exp\left(-\frac{r(1-\alpha)^2}{4}\right) <  x (1 - x)^{2m^2 - 2m}.$$
We set $x = \frac{1}{2m^2 - 2m + 1}$. Then, using the fact that $(\frac{y-1}{y})^{y-1}>\frac{1}{e}$ for $y \geq 1$, we may rewrite this inequality as 
$$ \exp\left(-\frac{r(1-\alpha)^2}{4}\right)  \leq \frac{1}{e (2m^2 - 2m+1)}.$$
Taking the logarithm of both sides and simplifying, we obtain
$$r \geq \frac{4(1 + \log(2m^2 - 2m + 1))}{(1-\alpha)^2}.$$
\end{proof}

Combining Theorem \ref{thm:technicalmanyPM} with Theorem \ref{thmLLLPM}, the following result is immediate. The condition $m \geq 37$ is included to ensure that the maximum degree of $G$ is greater than the minimum degree of each $G_i$.

\begin{theorem}
Let $m \geq 37$ be an integer, let $G$ be a graph on $2n$ vertices with maximum degree $m$, and let $\mathcal{G}~=~\{G_1, \dots, G_n\}$ be a family of subgraphs of $G$ such that each $G_i \in \mathcal{G}$ has minimum degree at least $t$. If $\mathcal{G}$ contains a perfect matching transversal
and
$$t \geq \frac{4(1 + \log(2m^2 - 2m + 1))}{(1-\alpha)^2}+1$$
for some $\alpha \in (0,1)$, then $\mathcal G$ contains at least $\lfloor \frac{1}{2}\alpha(t-1) +1\rfloor !$ distinct perfect matching transversals.
\label{thm:technicalmanyPM}
\end{theorem}

Observe that Theorem \ref{thmPM} follows from Theorem \ref{thm:technicalmanyPM} by plugging in $c = \frac{1}{10}$, letting $t$ be as small as possible, and adjusting the lower bound on $m$ to ensure the hypotheses are not vacuous.

We now turn our attention to the case where $G = K_{2n}$ and each graph $G_i \in \mathcal{G}$ has minimum degree at least $cn$ for some $c \geq 1$. Theorem 2 of \cite{JoosKim2020} states that, in this setting, $\mathcal{G}$ has a perfect matching transversal, say $(M,\phi)$. To obtain Theorem \ref{thmDiracPM}, let $\alpha = \left(1+\frac{\epsilon}{2}\right)^{-1}$ and observe that for sufficiently large $n$ we have $n \geq k_1 \log(k_2n^2)$ for any constants $k_1$ and $k_2$. We may therefore apply Theorem \ref{thm:technicalmanyPM} to obtain at least $\left\lfloor \frac{cn}{2+\epsilon} \right\rfloor !$ perfect matching transversals.

\section{Conclusion}
\label{secConclusion}

While we have obtained factorial lower bounds for the number of Hamiltonian transversals in certain graph families, it is unclear whether these bounds are tight. In \cite{Cuckler}, Cuckler and Kahn show that a graph on $n$ vertices with minimum degree at least $n/2$ must contain at least $(cn)^n$ distinct Hamiltonian cycles, with $c \approx \frac{1}{2e}$. Applying Stirling's approximation to Theorem \ref{thmDiracHam}, we see that in a graph family $\mathcal{G}$ on $n$ vertices in which each graph has minimum degree at least $n/2$, we can find $(\frac{c}{e}n)^{cn}$ Hamiltonian transversals, with $c \geq \frac{1}{68}$. Perhaps the most interesting open problem raised by the above work is whether it is possible to remove the constant $c<1$ from the exponent to obtain a lower bound for the number of  Hamiltonian transversals similar to that of Cuckler and Kahn.

We can also ask a similar question concerning families in which all graphs are regular. In \cite{Haythorpe} Haythorpe conjectures that for sufficiently large $k$, a Hamiltonian $k$-regular graph on $n$ vertices contains $\Omega(k^n)$ distinct Hamiltonian cycles. It would be also be interesting to try to prove or disprove a similar lower bound for the number of distinct Hamiltonian transversals over a $k$-regular graph family with a Hamiltonian transversal. Notice that this would be a significant improvement upon the bound in Theorem \ref{thmMain}.

In general, it would be interesting to find upper bounds for the number of possible Hamiltonian transversals in the types of graph families we have studied. We imagine that these results would come in the form of an explicit graph family which satisfies, say, the hypotheses of Theorem \ref{thmMain} and an upper bound on the number of Hamiltonian transversals in this family. Similar results for perfect matching transversals would also be very interesting, as would improvements on our lower bounds.

Using ideas similar to those presented in Section \ref{sectionRYB}, along with the local lemma, it is possible to prove a version of Lemma \ref{lemmaIndDom} in which a fixed set of edges in guaranteed to appear in both the original and the new Hamiltonian transversal. This could be used to improve upon Theorem \ref{thmMain}, but it seems that for such an improvement one would need to also ensure that certain edges do not appear in the second Hamiltonian transversal. It is also possible to prove a similar extension of Lemma \ref{lemmaPMcond}, and if it were possible to ensure that certain edges do not appear in the second perfect matching transversal one could likely improve upon the lower bound of Theorem \ref{thmPM}.

More generally, it would be interesting to adapt our combinatorial techniques to similar problems. Of course, these techniques were inspired by a paper of Thomassen  \cite{Thomassen} which considers a graph $G$ whose edges are colored red and green. Thomassen gives a combinatorial interpretation for a red-independent and green-dominating vertex set in terms of the number of Hamiltonian cycles in $G$. In Lemma \ref{lemmaIndDom}, we consider a graph $H$ whose edges are colored red, yellow, and blue, and we give a combinatorial interpretation for a vertex subset that is red-independent, yellow-dominating, and blue-dominating in terms the number of of Hamiltonian transversals in the graph family $\mathcal{G}$ from which $H$ was obtained. Perhaps by considering graphs whose edges are colored with many colors and vertex subsets that are red-independent and dominating in many colors, more combinatorial interpretations of this flavor can be found.

\raggedright
\bibliographystyle{abbrv}
\bibliography{Refs}

\end{document}